%% file: main.tex
\documentclass[letterpaper, 10 pt, conference]{ieeeconf}  

\IEEEoverridecommandlockouts                              

\usepackage{multicol}
\usepackage{scalerel}
\usepackage[utf8]{inputenc} 
\usepackage{adjustbox}
\usepackage{graphicx}
\usepackage{mathrsfs}
\usepackage{stmaryrd}
\usepackage{tikz}
\usepackage{amssymb,amsthm}
\usepackage[T1]{fontenc}
\usepackage{url}
\usepackage{ifthen}
\usepackage{cite}
\usepackage[cmex10]{amsmath}
\usepackage{pgfplots}
\usepackage{acronym}
\pgfplotsset{compat=newest}
\usepackage{geometry}
\usepackage{afterpage}

\DeclareMathOperator*{\Var}{\text{Var}}

\newcommand{\indep}{\perp \!\!\! \perp}
\newcommand{\brackets}[1]{\left(#1\right)}
\newcommand{\sbrackets}[1]{\left[#1\right]}

\newcommand{\mlt}[1]{\textcolor{black}{#1}}

\newtheorem{theorem}{Theorem}[section]
\newtheorem{corollary}[theorem]{Corollary}
\newtheorem{lemma}[theorem]{Lemma}
\newtheorem{proposition}[theorem]{Proposition}
\newtheorem{definition}[theorem]{Definition}
\newtheorem{remark}[theorem]{Remark}

\acrodef{mmse}[MMSE]{{minimum mean-squared error}}
\acrodef{dm}[DM]{{decision maker}}
\acrodef{lqg}[LQG]{Linear-Quadratic-Gaussian}
\acrodef{pdf}[p.d.f.]{probability density function}
\acrodef{dpc}[DPC]{Dirty Paper Coding}

\let\originalproof\proof
\let\originalendproof\endproof

\let\proof\originalproof
\let\endproof\originalendproof

\title{\LARGE \bf
Optimal Gaussian Strategies for Vector-valued Witsenhausen Counterexample with Non-causal State Estimator
}

\author{Mengyuan Zhao$^{1}$, Tobias J. Oechtering$^{1}$ and Maël Le Treust$^{2}$
\thanks{This work was supported by Swedish Research Council (VR) under grant 2020-03884.}
\thanks{$^{1}$M. Zhao and T. J. Oechtering are with the Division of Information Science and Engineering,
        KTH Royal Institute of Technology, 100 44 Stockholm, Sweden
        {\tt\small \{mzhao, oech\}@kth.se }}%
\thanks{$^{2}$M. Le Treust is with CNRS, Inria, IRISA UMR 6074, University of Rennes,
        F-35000 Rennes, France
        {\tt\small mael.le-treust@cnrs.fr}}%
}

\begin{document}

\maketitle
\thispagestyle{empty}
\pagestyle{empty}

\begin{abstract}

In this study, we investigate a vector-valued Witsenhausen model where the second \ac{dm} acquires a vector of observations before selecting a vector of estimations. Here, the first \ac{dm} acts causally whereas the second \ac{dm} estimates non-causally. When the vector length grows, we characterize, via a single-letter expression, the optimal trade-off between the power cost at the first \ac{dm} and the estimation cost at the second \ac{dm}.  In this paper, we show that the best linear scheme is achieved by using the time-sharing method between two affine strategies, which coincides with the convex envelope of the solution of Witsenhausen in 1968. Here also, Witsenhausen's two-point strategy and the scheme of Grover and Sahai in 2010 where both devices operate non-causally, outperform our best linear scheme. Therefore, gains obtained with block-coding schemes are only attainable if all \ac{dm}s operate non-causally.

\end{abstract}

\section{Introduction}
In 1968, Witsenhausen proposed his celebrated counterexample \cite{witsenhausen1968} showing that the optimal control law of the \ac{lqg} problem is not linear when the information pattern is nonclassical.  Nowadays, it still serves as an important toy example in distributed decision-making field \cite{bansal1986stochastic, silva2010control, yuksel2013stochastic, gupta2015existence} and information-theoretic control \cite{martins2005fundamental,
freudenberg2008feedback, 
derpich2012improved, AgrawalDLL15, Akyol2017information, charalambous2017hierarchical, wiese2018secure, 
Stavrou2022sequential}.


Despite the simplicity of Witsenhausen counterexample, finding its globally optimal strategy and optimal cost remains an open problem. To understand the difficulty, we consider a decentralized stochastic decision problem where the first \ac{dm} knows the state of the system perfectly but has a power cost. The second \ac{dm} estimates the interim state from a noisy observation which causes an estimation cost. Thus, the first DM aims to serve two purposes of control, steer the state at low cost as well as enable effective state estimation (dual role of control). Numerous studies over the last decades have aimed to enhance understanding and propose solutions using approaches such as numerical optimization \cite{tseng2017alocal, karlsson2011iterative}, neural networks \cite{baglietto2001numerical}, hierarchical search \cite{lee2001hierarchical}, learning \cite{li2009learning}, and optimal transport \cite{wu2011transport}.

Motivated by the vector-valued extension of the counterexample formulated in \cite{Grover2010Witsenhausen}, advanced coding schemes using methods from information theory provide new insights. For example, \cite{Grover2010Witsenhausen} established a lowerbound where both the \ac{dm}s have non-causal access to the entire sequences of 
 observation employing block-coding. Additionally, \cite{grover2013approximately} obtained a lattice-based optimal solution for finite-length vector case. The vector-valued approaches simplify the characterization of fundamental bounds in the limit of large vector lengths, additionally facilitating performance gains from block-coding through non-causal operations.

However, it remained open whether causal control schemes also could offer any advantage or even surpass the performance of Witsenhausen's two-point strategy. In \cite{Treust2024power}, the optimal Gaussian cost for the Witsenhausen problem is under study where the first \ac{dm} is non-causal and the second \ac{dm} is causal. This optimal Gaussian cost was achieved through a so-called \textit{time-sharing} strategy, which convexifies the linear cost region by utilizing two operational points. In this paper, we flip the causality property to explore the optimal Gaussian cost for the scenario where the first DM acts causally while the second DM acts non-causally. To this end, we employ the theoretical result of the characterization of the achievable Witsenhausen cost region and the information constraints derived in \cite{zhao2024coordination}. Despite differences with the cost region of \cite{Treust2024power}, both frameworks end up to have the same optimal Gaussian cost. Moreover, 
we uncover the remarkable finding that across five vector-valued Witsenhausen problem setups that feature at least one causal controller, there exists an identical optimal Gaussian cost outcome, regardless of the presence of any feedback or feed-forward information \cite{zhao2024causal}. Surprisingly, this optimal vector-valued Gaussian cost featuring causal DMs is again outperformed by Witsenhausen's two-point strategy \cite{witsenhausen1968} and the non-causal strategy by Grover and Sahai \cite{Grover2010Witsenhausen}.

This paper is structured as follows: Section \ref{sec: system model} introduces the model of causal encoding and noncausal decoding. The main result of the optimal Gaussian cost, a lemma   
determining the relation of Gaussian covariance coefficients given Markov chains and a corollary implied by the main result are presented in Section \ref{sec: opt gaussian}. Section \ref{sec: discussions} discusses the numerical results and the theorem's implications which serve as our main contribution. The proofs of the main theorem, supportive lemmas and corollary are shown in Appendix.

\section{System Model}\label{sec: system model}

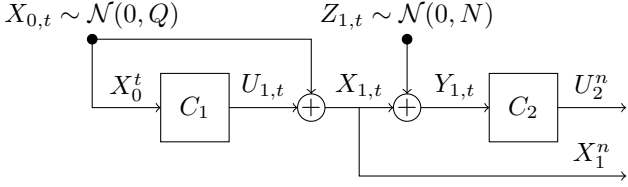
\begin{figure}[t]
  \centering
    \input{fig/Wits-Figure}
  \label{fig:sim}
  \caption{The i.i.d. state and the channel noise are drawn according to Gaussian distributions $X_0^{n}\sim \mathcal{N}(0,Q\mathbb{I})$ and $Z_1^{n}\sim \mathcal{N}(0,N\mathbb{I})$.}
\end{figure}

In this section, we introduce the setup and recapitulate some  foundational results to our problem. Throughout this paper, capital letters, e.g. $X_{0}$ denote random variables while lowercase letters, e.g. $x_0$ denote realisations. The notation $X_0^n$ denotes a random vector of length $n\in\mathbb N$, $X_{0,t}$ denotes the $t$-th entry of $X_0^n$, and $X_0^t=(X_{0,1},...,X_{0,t})$ represents the segment of $X_0^n$ up to stage $t$, where $t\in\{1,\ldots,n\}$.

Let's consider the vector-valued Witsenhausen counterexample setup with causal source states and channel noises that are drawn independently according to the i.i.d. Gaussian distributions $X_0^n\sim\mathcal{N}(0, Q\mathbb I)$ and $Z_1^n\sim\mathcal{N}(0,N\mathbb I)$, for some $Q,N\in\mathbb R^+$, where $\mathbb I$ is the identity matrix, see Figure \ref{fig:sim}. We denote by $X_1$ the memoryless interim state and $Y_1$ the output of the memoryless additive channel, generated by
\begin{flalign}
    &X_1 = X_0 + U_1  &\text{with }X_0\sim\mathcal{N}(0,Q),\label{X1 generation}\\
    &Y_1 = X_1 + Z_1 = X_0 + U_1 + Z_1 &\text{with }Z_1\sim\mathcal{N}(0,N).\label{Y1 generation}
\end{flalign}
We denote by $\mathcal{P}_{X_0} = \mathcal{N}(0,Q)$ the generative Gaussian probability distribution of the state, and by $\mathcal{P}_{X_1,Y_1|X_0,U_1}$ the channel probability distribution according to \eqref{X1 generation} and \eqref{Y1 generation}.

\begin{definition}
    For $n\in\mathbb{N}$, a ``control design'' with causal encoder and noncausal decoder is a tuple of conditional distributions $c = (\{ f_{U_{1,t}|X_0^t}^{(t)}\}_{t=1}^n, g_{U_2^n|Y_1^n})$, where at time instant $t\in\{1,...,n\}$, $f_{U_{1,t}|X_0^t}^{(t)}$ selects a channel input $U_{1,t}$ based on the past source sequence $X_0^t$ up to $t$, while $g_{U_2^n|Y_1^n}$ selects the whole estimation sequence $U_2^n$ based on the whole channel output sequence $Y_1^n$.      
    This induces a distribution over sequences of symbols:
    \begin{equation}
    \prod_{t=1}^n \mathcal{P}_{X_{0,t}}\times \prod_{t=1}^n \mlt{f_{U_{1,t}|X_0^t}^{(t)}}\times\prod_{t=1}^n \mathcal{P}_{X_{1,t},Y_{1,t}|X_{0,t},U_{1,t}} \times \mlt{g_{U_2^n|Y_1^n}}, \nonumber
    \end{equation}
We denote by $\mathcal{C}_e(n)$ the set of control designs with causal encoder and non-causal decoder.
\end{definition} 

We evaluate the power cost and the estimation cost by considering their respective average over the sequences of symbols.
\begin{definition}
    We define the two $n$-stage cost functions $c_P(u_1^n) = \frac{1}{n}\sum_{t=1}^n (u_{1,t})^2$ and $c_S(x_1^n, u_2^n) = \frac{1}{n}\sum_{t=1}^n(x_{1,t}-u_{2,t})^2$. 
    The pair of costs $(P,S)\in\mathbb{R}^2$ is achievable if for all $\varepsilon>0$, there exists $\Bar{n}\in\mathbb N$ such that for all $n\geq \Bar{n}$, there exists a control design $c\in \mathcal{C}_e(n)$ such that 
    \begin{equation}
        \mathbb E\Big[\big|P - c_P(U_1^n)\big| + \big|S - c_S(X_1^n, U_2^n)\big|\Big] \leq \varepsilon.\nonumber
    \end{equation}
\end{definition}
 The optimal achievable pairs of costs $(P,S)\in\mathbb R^2$, which we refer to as the Witsenhausen cost, are characterized in the following theorem.

\begin{theorem}[\!\!\!\protect{\cite[Theorem I]{zhao2024coordination}}]
    The pair of Witsenhausen costs $(P,S)$ is achievable if and only if there exists a joint distribution over the random variables $(X_0, W_1, W_2, U_1, X_1, Y_1, U_2)$ that decomposes according to
    \begin{flalign}        \mathcal{P}_{X_0}\mathcal{P}_{W_1}\mathcal{P}_{W_2|X_0,W_1}\mathcal{P}_{U_1|X_0,W_1}\mathcal{P}_{X_1, Y_1|X_0, U_1}\mathcal{P}_{U_2|W_1, W_2, Y_1},\label{eq: prob result}
    \end{flalign}
    such that
    \begin{align}
        &I(W_1, W_2; Y_1) - I(W_2; X_0 | W_1) \geq 0,\label{eq: info result}\\
        &P = \mathbb{E}\sbrackets{U_1^2}, \quad \quad S = \mathbb{E}\sbrackets{(X_1 - U_2)^2},\label{eq: cost result}
    \end{align}
    where $\mathcal{P}_{X_0}$ and $\mathcal{P}_{X_1, Y_1|X_0, U_1}$ are the given Gaussian distributions,  $W_1,W_2$ are two auxiliary random variables, and where the mutual information $I(W_1, W_2; Y_1)$ is the Kullback-Leibler divergence between the joint distribution $\mathcal{P}_{W_1, W_2, Y_1}$ and the product of the marginal distributions $\mathcal{P}_{W_1, W_2} \mathcal{P}_{Y_1}$.\label{wits main theorem}
\end{theorem}

Condition \eqref{eq: info result} characterizes the feasibility set of the Witsenhausen costs. The auxiliary random variables $W_1,W_2$ are introduced specifically for forming the single-letter solutions \eqref{eq: prob result} and \eqref{eq: info result}. They have operational meanings: $W_1$ is used for codewords adapted to the channel assisting a reliable communication, and $W_2$ is used for a description of the compressed source states. Therefore, the dual role of control is explicitly captured by the two auxiliary random variables. Since the region characterized above is optimal, $W_1,W_2$ also provide all freedom for the optimization process.

\begin{remark}
    The following Markov chains follow from the joint probability distribution \eqref{eq: prob result}:
    \begin{align}
        \left\{  
        \begin{aligned}
            &X_0\text{ is independent of }W_1,\\
            & U_1 -\!\!\!\!\minuso\!\!\!\!- (X_0, W_1) -\!\!\!\!\minuso\!\!\!\!- W_2,\\
            &(X_1, Y_1)-\!\!\!\!\minuso\!\!\!\!- (X_0, U_1)    -\!\!\!\!\minuso\!\!\!\!- (W_1, W_2),\\
            & U_2 -\!\!\!\!\minuso\!\!\!\!- (W_1, W_2, Y_1) -\!\!\!\!\minuso\!\!\!\!- (X_0, U_1, X_1).
        \end{aligned}
        \right.
        \label{markov result}
    \end{align}
\end{remark}
The first two Markov chains are consequences of causal encoding. The third Markov chain is related to the processing order of the Gaussian channel. The last Markov chain comes from the non-causal decoding and the symbol-wise reconstruction. These Markov chains play a crucial role in the proof of the main theorem.

\section{Optimal Gaussian Cost}\label{sec: opt gaussian}

In the following, we fix a power cost $P\geq 0$ and investigate the optimal estimation cost at the decoder obtained from using Gaussian random variables.

\begin{definition}
    Given a power cost parameter $P\geq 0$, we define the estimation cost obtained by jointly Gaussian random variables to be
        \begin{align}
            &S_{\mathsf {G}}(P) = \inf_{\mathcal{P}\in\mathbb P(P)}\mathbb E\Big[\big(X_1 - U_2\big)^2\Big],\label{eq: opt gauss prob}\\
            &\mathbb P(P) = \Bigl\{ (\mathcal{P}_{W_1},\mathcal{P}_{W_2|X_0 ,W_1},\mathcal{P}_{U_1|X_0 ,W_1}, \mathcal{P}_{U_2|W_1 ,W_2 ,Y_1}),\nonumber\\
            &\quad \text{ s.t. }P=\mathbb{E}\big[U_1^2\big],\; I(W_1,W_2;Y_1) - I(W_2;X_0|W_1)\geq0,\nonumber\\
            &\quad X_0,W_1,W_2,U_1,X_1,Y_1,U_2\text{ are jointly Gaussian}\Bigr\}.
            \label{eq: opt domain w/o feedback}
        \end{align}
The set $\mathbb P(P)$ denotes the optimization domain.
    
\end{definition}
Note that the \ac{mmse} estimation for the decoder is given by the conditional expectation. We have the following proposition.
\begin{proposition}
    Given a power cost parameter $P\geq 0$, the estimation cost $S_{\mathsf {G}}(P)$ satisfies   
        \begin{align}
            &S_{\mathsf {G}}(P) = \inf_{\mathcal{P}\in\mathbb P_{\mathsf {G}}(P)}\mathbb E\Big[\big(X_1 - \mathbb E\big[X_1\big|W_1,W_2,Y_1\big]\big)^2\Big],\label{eq: opt prob}\\
            &\mathbb P_{\mathsf {G}}(P) = \Bigl\{  (\mathcal{P}_{W_1},\mathcal{P}_{W_2|X_0, W_1},\mathcal{P}_{U_1|X_0, W_1})\quad \text{s.t. }P = \mathbb E\big[U_1^2\big],\nonumber\\
            &\quad I(W_1,W_2;Y_1) - I(W_2;X_0|W_1)\geq0,\text{ and }\\
            &\quad X_0,W_1,W_2,U_1,X_1,Y_1,U_2\text{ are jointly Gaussian}\Bigr\}.\nonumber
        \end{align}
\end{proposition}

The best linear strategy below in our setting is again the same as the one in Witsenhausen's paper \cite{witsenhausen1968} following the same arguments.

\begin{lemma}[\!\!\protect{\cite[Lemma 11]{witsenhausen1968}}, \protect{\cite[Lemma 5]{Treust2024power}}]
    The best linear policy is $U_1=-\sqrt{\frac{P}{Q}}X_0$, if $P\leq Q$, otherwise $U_1=-X_0 + \sqrt{P-Q}$, which induces the estimation cost
    \begin{align}
        S_{\ell}(P)=\begin{cases}\frac{(\sqrt{Q}-\sqrt{P})^{2}\cdot N}{(\sqrt{Q}-\sqrt{P})^{2}+N} & \text{ if } P\in[0,\ Q],\\ 0 &\text{ otherwise } .\end{cases}\label{eq: opt linear cost}
    \end{align}
\end{lemma}
The best linear cost provides an upper bound for the optimal Gaussian cost. However, the function $P\mapsto S_{\ell}(P)$ is not always convex. The best Gaussian estimation cost derived in Theorem \ref{thm: best gaussian policy} below obtains the convex envelope of $S_{\ell}$, which coincides with the solution proposed by Witsenhausen in \cite[Lemma 12]{witsenhausen1968} and the solution for the problem of flipped causality \cite[Theorem 2]{Treust2024power}. 

The proofs of Theorem \ref{thm: best gaussian policy}, Lemma \ref{lemma: markov-covariance} and Corollary \ref{cor: causal cost with feedback} stated in the following are shown in the appendix.

\begin{theorem}[Main Result]\label{thm: best gaussian policy}
    The optimal Gaussian estimation cost for Witsenhausen problem with causal encoder and non-causal decoder is given by 
    \begin{flalign} S_{\mathsf {G}}(P) =& \begin{cases} \frac {N\cdot (Q-N-P)}{Q} & \text {if } Q> 4N \text { and } P\in [P_{1},P_{2}],\\ S_{\ell }(P) & \text {otherwise. } \end{cases} \label{opt gaussian cost}\end{flalign}
    where the parameters
    \begin{align}
        P_1 &= \frac{1}{2}(Q-2N-\sqrt{Q^2-4QN}),\label{eq: p1}\\
        P_2 &= \frac{1}{2}(Q-2N+\sqrt{Q^2-4QN}).\label{eq: p2}
    \end{align}
\end{theorem}
In this theorem, $P_1$ and $P_2$ given in \eqref{eq: p1} and \eqref{eq: p2} are the two operating points for conducting the time-sharing strategy, which enables us to achieve a cost gain from the affine policy.

The following lemma states the general relation of Gaussian covariance coefficients given a Markov chain. It is a direct consequence combining several well-known results.


\begin{lemma}\label{lemma: markov-covariance}
    If the jointly Gaussian random vector $(X,Y,Z)$ satisfy the Markov chain $X -\!\!\!\!\minuso\!\!\!\!- Y -\!\!\!\!\minuso\!\!\!\!- Z$ and have a covariance matrix
    \begin{align}
        \Sigma_{X,Y,Z} = \begin{pmatrix}
            P & \rho_1\sqrt{PQ} & \rho_2\sqrt{PV}\\
            \rho_1\sqrt{PQ} & Q &\rho_3\sqrt{QV}\\
            \rho_2\sqrt{PV}&\rho_3\sqrt{QV} & V
        \end{pmatrix},\label{eq: cov xyz}
    \end{align}
    with the covariance coefficients $(\rho_1,\rho_2,\rho_3)\in[-1,1]^3$ ensuring that $\det\brackets{\Sigma_{X,Y,Z}}\geq 0$, then, we have
    \begin{align}
        \rho_2=\rho_1\rho_3.\label{eq: rho_2=rho_1rho_3}
    \end{align}
    
\end{lemma}

In other words, with the context of the Markov chain of jointly Gaussian $X -\!\!\!\!\minuso\!\!\!\!- Y -\!\!\!\!\minuso\!\!\!\!- Z$, if $X$ and $Y$ (or if $Z$ and $Y$) are uncorrelated (i.e., if $\rho_1=0$ or $\rho_3=0$), it follows that $X$ and $Z$ are also uncorrelated (i.e., $\rho_2=0$).

\begin{corollary}\label{cor: causal cost with feedback}
    The problem for causal encoder and non-causal decoder with channel feedback (i.e., channel output $Y_1$ is available to the first DM) has the same cost result \eqref{opt gaussian cost}.
\end{corollary}

The cost region characterization of this setup is investigated in \cite[Sec. C]{zhao2024causal}. From Corollary \ref{cor: causal cost with feedback}, we get that having channel feedback information to assist communication does not contribute to any performance gain in the Gaussian case.

Next, we recall Witsenhausen's two-point strategy. It outperforms the best Gaussian cost \eqref{opt gaussian cost} for some values of $Q$ and $N$.
\begin{theorem}[\protect{\cite[Sec. 6]{witsenhausen1968}}, \protect{\cite[Prop. 11]{Treust2024power}}]
    For parameter $a\geq 0$, Witsenhausen's two-point strategy is given by
    \begin{align*}
        U_1 = a\cdot \text{sign}(X_0) - X_0.
    \end{align*}
    The power and estimation costs are given by
    \begin{align}
        P_2(a) &= Q + a\brackets{a - 2\sqrt{\frac{2Q}{\pi}}},\nonumber\\
        S_2(a) &= a^2\sqrt{\frac{2\pi}{N}}\phi\brackets{\frac{a}{\sqrt{N}}}\int \frac{\phi\brackets{\frac{y_1}{\sqrt{N}}}}{\cosh{(\frac{ay_1}{N})}}dy_1,\label{eq: two-point cost}
    \end{align}
    where $\phi(x) = \frac{1}{\sqrt{2\pi}}e^{-\frac{x^2}{2}}$ and the optimal receiver's strategy is given by $\mathbb E\big[X_1\big|Y_1 = y_1\big]=a\tanh{(\frac{ay_1}{N})}$.
\end{theorem}

\section{Discussions}\label{sec: discussions}

    The numerical results of the best affine cost $S_\ell(P)$ in \eqref{eq: opt linear cost}, the optimal Gaussian cost $S_{\mathsf{G}}(P)$ in \eqref{opt gaussian cost}, Witsenhausen's two-point strategy cost $S_2(P)$ in \eqref{eq: two-point cost}, and the Dirty Paper Coding (DPC) based cost for cases where both \ac{dm}s are non-causal $S_{\mathsf{dpc}}(P)$ in \cite[App. D.1-D.7]{Grover2010Witsenhausen} and  \cite[Eq.(47)]{Treust2024power} are illustrated in Fig \ref{fig:mmse_curve}, for the parameters $(Q,N)=(0.8,0.1)$.
    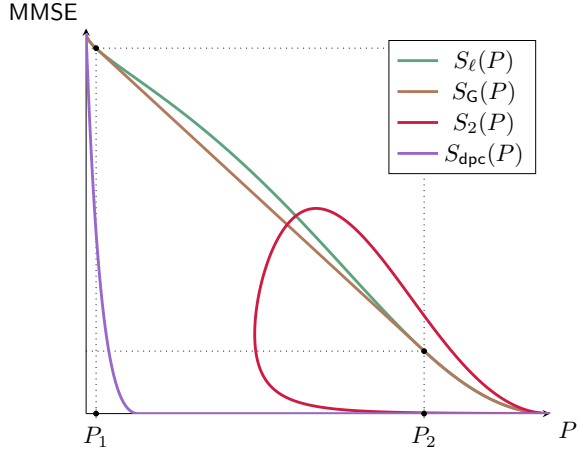
\begin{figure}[t]
        \centering
        \input{fig/MMSE_curves.tex}
        \caption{Comparison of the four cost functions $S_\ell(P)$, $S_{\mathsf{G}}(P)$, $S_2(P)$ and $S_{\mathsf{dpc}}(P)$. In this particular case, $S_2(P)$ outperforms $S_{\mathsf{G}}(P)$ and $S_\ell(P)$, and the cost induced by the non-causal strategy $S_{\mathsf{dpc}}(P)$ outperforms all other cost functions. }
        \label{fig:mmse_curve}
    \end{figure}


\begin{remark}Now, we discuss the optimal Gaussian cost for the case where both \ac{dm}s are causal.
\begin{itemize}
    \item The optimal Gaussian strategy discussed in \cite[App. B-B]{Treust2024power} is a single-letter (causal) approach, which provides a valid control strategy for a more restrictive setup assuming that the second \ac{dm} is also causal.  Therefore, \eqref{opt gaussian cost} is an upper bound of this cost. 
    \item From the genie-aided argument, a lower bound for the optimal cost is given by the optimal cost of a more superior system, such as the setup discussed in Corollary \ref{cor: causal cost with feedback}. Therefore, \eqref{opt gaussian cost} is a lower bound of the optimal cost of causal encoding and causal decoding Witsenhausen setup. 
\end{itemize}
Consequently, the optimal Gaussian cost for the causal encoding and decoding setup is the same as \eqref{opt gaussian cost}.
\end{remark}

    \begin{remark}
Our discussions above lead us to a remarkable conclusion: Even though the optimal achievable cost regions are different, the optimal Gaussian costs for the following five distinct configurations are the same:
    \begin{enumerate}
        \item Causal encoding and causal decoding.
        \item Causal encoding and non-causal decoding.
        \item Causal encoding and non-causal decoding with channel feedback.
        \item Non-causal encoding and causal decoding.
        \item Non-causal encoding and causal decoding with source feedforward.
    \end{enumerate}
    This equality is surprising since these settings all have different single-letter optimal regions, but these regions coincide if we restrict to Gaussian schemes. Therefore, block-coding gains can be obtained only if both \ac{dm}s are non-causal.

    
\end{remark}

\section*{Appendix}\label{sec: appendix}

\begin{proof}[Proof of the Main Result]


Without loss of generality, we consider the joint Gaussian random variables $(X_0, W_1,W_2, U_1)\sim\mathcal{N}(0, K)$ optimal for problem \eqref{eq: opt prob}, are centered with the covariance matrix 
\begin{flalign}
    K&= \begin{pmatrix}
        Q & \rho_1\sqrt{QV_1} & \rho_2\sqrt{QV_2} & \rho_3\sqrt{QP}\\
         \rho_1\sqrt{QV_1} & V_1 &\rho_4\sqrt{V_1V_2} & \rho_5\sqrt{V_1P}\\
         \rho_2\sqrt{QV_2}&\rho_4\sqrt{V_1V_2}&V_2&\rho_6\sqrt{V_2P} \\
         \rho_3\sqrt{QP} & \rho_5\sqrt{V_1P} & \rho_6\sqrt{V_2P} & P
    \end{pmatrix}.\label{eq: cov x_0, w_1, w_2, u_1}
\end{flalign}
Since $X_0\indep W_1$, $\rho_1 = 0$. Also, given the Markov chain $U_1 -\!\!\!\!\minuso\!\!\!\!- (X_0, W_1) -\!\!\!\!\minuso\!\!\!\!- W_2$, from Lemma \ref{lemma: markov-covariance}, we can obtain $\rho_6 =\rho_2\rho_3 + \rho_4\rho_5 $. Moreover, other active correlation coefficients $(\rho_2, \rho_3, \rho_4, \rho_5)\in[-1,1]^4$ are chosen such that 
\begin{align*}
   \det(K) =QV_1V_2P(-1 + \rho_2^2 + \rho_4^2)(-1 + \rho_3^2 + \rho_5^2)\geq 0.
\end{align*}

Given \eqref{eq: cov x_0, w_1, w_2, u_1} and \eqref{Y1 generation}, the covariance matrix $K_2$ of $(X_0,W_1,W_2,Y_1)$ could be easily computed, with a determinant given by
\begin{align*}
    \det(K_2) = QV_1V_2(-1 + \rho_2^2 + \rho_4^2)(P(-1 + \rho_3^2 + \rho_5^2) - N).
\end{align*}
The positive semi-definite property of $K_2$ must also be satisfied with properly chosen $(\rho_2, \rho_3, \rho_4, \rho_5)$.

We have the following lemma determining the explicit formulas of the information constraint \eqref{eq: info result} and the optimization object \eqref{eq: opt prob}.

\begin{lemma}\label{lemma, eqs to show}
    Assume $(X_0, W_1,W_2, U_1)\sim\mathcal{N}(0, K)$, then
    \begin{align}
        &I(W_1, W_2; Y_1) - I(W_2; X_0|W_1)\nonumber\\
        &= I(W_1; Y_1) - I(W_2; X_0|W_1, Y_1)\label{eq: info constraint, w/o feedback}\\
        &= \frac{1}{2}\log\brackets{\frac{T_1}{T_1 - T_2}},\label{eq: new info constraint, w/o feedback}
    \end{align}
    where the terms 
    \begin{align}
        &T_1 = (P+Q+N+2\rho_3\sqrt{QP})(-1 + \rho_2^2 + \rho_4^2),\nonumber\\
        &T_2 = N\rho_2^2 + P\rho_2^2(1-\rho_3^2) - P\rho_5^2(1-\rho_4^2).\nonumber
    \end{align}
    And the object to minimize is
    \begin{align}
        &\mathbb E\Big[\big(X_1 - \mathbb E\big[X_1\big|W_1,W_2,Y_1\big]\big)^2\Big]\nonumber\\
        &=\frac{N\cdot f_1(\rho_2,\rho_3,\rho_4,\rho_5)}{(1-\rho_4^2)\cdot N + f_1(\rho_2,\rho_3,\rho_4,\rho_5)},\label{eq: conditional variance}
    \end{align}
    where $f_1(\rho_2,\rho_3,\rho_4,\rho_5) = -P\rho_2^2\rho_3^2 - (Q+2\rho_3\sqrt{PQ})(-1+\rho_2^2+\rho_4^2) + P(1-\rho_4^2)(1-\rho_5^2)$.
\end{lemma}

\begin{proof}[Proof of Lemma \ref{lemma, eqs to show}]
     \begin{align*}
    &\quad I(W_1, W_2; Y_1) - I(W_2; X_0|W_1)\\
    &= I(W_1; Y_1) + I(W_2; Y_1|W_1) - I(W_2; X_0|W_1)\\
    &\stackrel{\text{(a)}}{=} I(W_1; Y_1) + I(W_2; Y_1|W_1) - I(W_2; X_0, Y_1|W_1)\\
    &= I(W_1; Y_1) - I(W_2; X_0|W_1, Y_1)\\
    & = \frac{1}{2}\log\brackets{\frac{\sigma_{W_1}^2\cdot\sigma_{Y_1}^2\cdot\det{(K_2)}}{\det{(\Sigma_{W_1,W_2,Y_1})}\cdot\det{(\Sigma_{X_0,W_1,Y_1})}}}\\
    & \stackrel{}{=} \frac{1}{2}\log\brackets{\frac{T_1}{T_1 - T_2}},
\end{align*}
where (a) comes from the Markov chain $Y_1 -\!\!\!\!\minuso\!\!\!\!- (X_0, W_1) -\!\!\!\!\minuso\!\!\!\!- W_2$, and thus $I(W_2; Y_1|X_0, W_1)=0$. Additionally,
\begin{align}
    &\mathbb E\Big[\big(X_1 - \mathbb E\big[X_1\big|W_1,W_2,Y_1\big]\big)^2\Big]\nonumber\\
    &= \Var\brackets{X_1|W_1, W_2, Y_1}\nonumber\\
    &\stackrel{\text{(b)}}{=} \sigma_{X_1}^2 - \Sigma_{X_1W}^\top\Sigma_{WW}^{-1}\Sigma_{X_1W}\nonumber\\
    &= \frac{N\cdot f_1(\rho_2,\rho_3,\rho_4,\rho_5)}{(1-\rho_4^2)\cdot N + f_1(\rho_2,\rho_3,\rho_4,\rho_5)},\nonumber
\end{align}
where step (b) is obtained using the Schur complement. Here, $W = (W_1,W_2,Y_1)^\top$, $\Sigma_{X_1W} = (\sigma_{X_1,W_1}, \sigma_{X_1,W_2}, \sigma_{X_1,Y_1})^\top$, and $\Sigma_{WW}$ is the covariance matrix of $(W_1,W_2,Y_1)$.
\end{proof}

Thus, given the expression of \eqref{eq: new info constraint, w/o feedback}, we can obtain a new expression of the original information constraint
\begin{align*}
    \frac{1}{2}\log\brackets{\frac{T_1}{T_1 - T_2}}\geq 0\Leftrightarrow T_1\geq T_2\geq 0 \text{ or } T_1\leq T_2\leq 0.
\end{align*}

If $1-\rho_4^2\ = 0$, from \eqref{eq: conditional variance}, we can get that 
\begin{align*}
    \mathbb E\Big[\big(X_1 - \mathbb E\big[X_1\big|W_1,W_2,Y_1\big]\big)^2\Big]=N.
\end{align*} 

Next, we focus on the case if $1-\rho_4^2\ \neq 0$. In this case, \eqref{eq: conditional variance} is of the form 
\begin{align}
    \frac{N\cdot f(\rho_2,\rho_3,\rho_4,\rho_5)}{N + f(\rho_2,\rho_3,\rho_4,\rho_5)}, \label{eq: Nf/(N+f)}
\end{align} 
where $f(\rho_2,\rho_3,\rho_4,\rho_5) = f_1(\rho_2,\rho_3,\rho_4,\rho_5)/(1-\rho_4^2) =  \frac{ -P\rho_2^2\rho_3^2 - (Q+2\rho_3\sqrt{PQ})(-1+\rho_2^2+\rho_4^2) + P(1-\rho_4^2)(1-\rho_5^2)}{(1 -\rho_4^2)}$.

Note that, the function $x\mapsto \frac{N\cdot x}{N + x}$ is nonnegative and strictly increasing over the region $(-\infty, -N]\cup [0, \infty)$. Therefore, our goal of minimizing \eqref{eq: Nf/(N+f)} is now transformed to either minimizing the nonnegative object 
\begin{align*}
 f(\rho_2,\rho_3,\rho_4,\rho_5)\geq 0
\end{align*} or minimizing the negative object
\begin{align*}
   f(\rho_2,\rho_3,\rho_4,\rho_5)\leq -N
\end{align*} subject to the following constraints:
 \begin{align}
        &\text{1.  }\det(K)\geq 0 \Longrightarrow\nonumber\\
        &\quad QV_1V_2P(-1 + \rho_2^2 + \rho_4^2)(-1 + \rho_3^2 + \rho_5^2)
        \stackrel{\text{(A)}}{\geq} 0,\nonumber\\
        &\text{2.  }\det(K_2)\geq 0 \Longrightarrow\nonumber\\
         &\quad QV_1V_2(-1 + \rho_2^2 + \rho_4^2)(P(-1 + \rho_3^2 + \rho_5^2) - N)
        \stackrel{\text{(B)}}{\geq} 0,\nonumber\\
        &\text{3.  }T_1\geq T_2 \geq 0\Longrightarrow\nonumber\\&\quad(Q+P+N+2\rho_3\sqrt{QP})(-1  + \rho_2^2 + \rho_4^2) \nonumber\\
        &\quad\stackrel{\text{(C1)}}{\geq} N\rho_2^2 + P\rho_2^2(1-\rho_3^2) - P\rho_5^2(1-\rho_4^2)\stackrel{\text{(D1)}}{\geq}0,\nonumber\\
        &\quad\text{or,   }T_1\leq T_2 \leq 0\Longrightarrow\nonumber\\&\quad(Q+P+N+2\rho_3\sqrt{QP})(-1  + \rho_2^2 + \rho_4^2) \nonumber\\
        &\quad\stackrel{\text{(C2)}}{\leq}  N\rho_2^2+ P\rho_2^2(1-\rho_3^2) - P\rho_5^2(1-\rho_4^2)\stackrel{\text{(D2)}}{\leq}0.\nonumber
    \end{align}

To simplify the above constraints, we consider the following two distinct cases:

\textit{Case 1, }if $-1 + \rho_2^2 + \rho_4^2\geq 0$, constraints (A) and (B) together yield $-1 + \rho_3^2 + \rho_5^2\geq \frac{N}{P}$. Moreover, constraint (C1) gives us $f(\rho_2,\rho_3,\rho_4,\rho_5) \leq -N$. In this case, our optimization problem boils down to minimizing
\begin{align}
    f(\rho_2,\rho_3,\rho_4,\rho_5)\leq -N,\label{eq: opt object, case 1}
\end{align}
subject to 
\begin{align}
    &\text{1.  }-1 + \rho_2^2 + \rho_4^2\geq 0,\label{eq: case 1, constraint 1}\\
    &\text{2.  }-1 + \rho_3^2 + \rho_5^2\geq \frac{N}{P},\label{eq: case 1, constraint 2}\\
    &\text{3.  } N\rho_2^2 + P\rho_2^2(1-\rho_3^2) - P\rho_5^2(1-\rho_4^2)\geq0.\label{eq: case 1, constraint 3}
\end{align}
Notice that $f(\rho_2,\rho_3,\rho_4,\rho_5)$ is decreasing function of $\rho_5^2$. From \eqref{eq: case 1, constraint 3}, we get that $\rho_5^2 \leq \frac{N\rho_2^2 + P\rho_2^2(1-\rho_3^2)}{P(1-\rho_4^2)}$. Therefore, the optimizer is given by $(\rho_5^*)^2 = \frac{N\rho_2^2 + P\rho_2^2(1-\rho_3^2)}{P(1-\rho_4^2)}$. By plugging $(\rho_5^*)^2$ into \eqref{eq: opt object, case 1}, we obtain that
\begin{align*}
    &f(\rho_2,\rho_3,\rho_4,\rho_5^*) \\
    &= \frac{-(-1+\rho_2^2+\rho_4^2)(Q+P+2\rho_3\sqrt{PQ}) - N\rho_2^2}{1-\rho_4^2}.
\end{align*}
Since
\begin{align*}
    \frac{\partial f(\rho_2,\rho_3,\rho_4,\rho_5^*)}{\partial \rho_4^2} \leq 0,
\end{align*}
we know that $f$ decreases to $-\infty$ as $\rho_4^2$ approaches $1$. Moreover, since $\mathbb E\big[\big(X_1 - \mathbb E\big[X_1\big|W_1,W_2,Y_1\big]\big)^2\big]$  is continuous and converges to $N$ when $\rho_4^2\rightarrow 1$, in this case, the minimal value of $N$ is obtained at the boundary $\rho_4^2 = 1$.

\textit{Case 2, }If $-1 + \rho_2^2 + \rho_4^2\leq 0$, conditions (A) and (B) together give us $-1 + \rho_3^2 + \rho_5^2\leq 0$, and (C2) gives us $f(\rho_2,\rho_3,\rho_4,\rho_5) \geq -N$ (but only $f\geq 0$ contributes to a nonnegative estimation cost). Therefore, in this case, our optimization problem boils down to minimizing 
\begin{align}
    f(\rho_2,\rho_3,\rho_4,\rho_5) \geq0,\label{eq: opt object, case 2}
\end{align}
subject to
\begin{align}
        &\text{1.  }-1 + \rho_2^2 + \rho_4^2\leq 0,\label{eq: case 2, constraint 1}\\
        &\text{2.  }-1 + \rho_3^2 + \rho_5^2\leq 0,\label{eq: case 2, constraint 2}\\
        &\text{3.  }N\rho_2^2+ P\rho_2^2(1-\rho_3^2) - P\rho_5^2(1-\rho_4^2)\leq 0. \label{eq: case 2, constraint 3}
\end{align}
Since $f(\rho_2,\rho_3,\rho_4,\rho_5)$ is reduced especially when $\rho_5^2$ is increased, therefore, from \eqref{eq: case 2, constraint 2}, we get that $(\rho_5^*)^2 = 1-\rho_3^2$. By replacing $(\rho_5^*)^2$ into \eqref{eq: opt object, case 2}, we get that
\begin{align}    f(\rho_2,\rho_3,\rho_4,\rho_5^*) =\frac{ (1 - \rho_2^2 - \rho_4^2)(\sqrt{P}\rho_3 + \sqrt{Q})^2 }{(1 -\rho_4^2)}.\label{eq: opt object 2}
\end{align}
Therefore, when $P\geq Q$, taking $\rho_3^* = -\sqrt{\frac{Q}{P}}$ and any $\rho_2$, $\rho_4$ satisfy the constraints results in the optimal value of  $f(\rho_2,\rho_3^*,\rho_4,\rho_5^*) = 0$. In this case, $\mathbb E\big[\big(X_1 - \mathbb E\big[X_1\big|W_1,W_2,Y_1\big]\big)^2\big] = 0$.

When $P<Q$, $f(\rho_2,\rho_3,\rho_4,\rho_5^*)$ is a decreasing function of $\rho_2^2$. The constraint \eqref{eq: case 2, constraint 3} gives us the optimal value of $\rho_2^2$:
\begin{equation}
    (\rho_2^*)^2 = \frac{P(1-\rho_3^2)(1-\rho_4^2)}{N+P(1-\rho_3^2)}.\label{opt rho_2^2}
\end{equation}
Plugging the $(\rho_2^*)^2$ in \eqref{opt rho_2^2} into 
\eqref{eq: opt object 2}, we have
\begin{align*}
    f(\rho_2^*,\rho_3,\rho_4,\rho_5^*)  = \frac{N (\sqrt{P}\rho_3 + \sqrt{Q})^2}{N+P(1-\rho_3^2)}.
\end{align*}
Then, taking $\frac{\partial f}{\partial \rho_3} =0$ gives us the optimum 
\begin{align*}
    \rho_3^* = -\frac{P+N}{\sqrt{QP}},
\end{align*}
which is valid only if the following condition holds
\begin{align}
    (\rho_3^*)^2 = \frac{(P+N)^2}{QP}\leq 1
    \Longrightarrow \quad\left\{
       \begin{aligned}
           & Q > 4N,\\
           & P\in\sbrackets{P_1, P_2},
       \end{aligned}
       \right. \label{eq: opt condition, case 2}
\end{align}
where $P_1 = \frac{1}{2}\brackets{Q-2N-\sqrt{Q^2-4QN}}$, $P_2 = \frac{1}{2}\brackets{Q-2N+\sqrt{Q^2-4QN}}$.
In this case, $(\rho_5^*)^2 = \frac{QP - (P+N)^2}{QP}$, and $f(\rho_2^*,\rho_3^*,\rho_4^*,\rho_5^*) = \frac{N(Q - N - P)}{N+P}$, which results in the estimation cost of
\begin{align*}
    \mathbb E\Big[\big(X_1 - \mathbb E\big[X_1\big|W_1,W_2,Y_1\big]\big)^2\Big] = \frac{N(Q-N-P)}{Q}.
\end{align*}

In the case when the condition \eqref{eq: opt condition, case 2} is unmet, we always have $-\frac{P+N}{\sqrt{QP}} <-1$. Since $\frac{\partial f}{\partial \rho_3}>0$, function $f$ increases when $\rho_3\in[-1,1]$ increases. Therefore, the minimal value of $f$ achieves at the left boundary $\rho_3^* = -1$, which gives us $\rho_2^* = \rho_5^* = 0$ and $f(\rho_2^*,\rho_3^*,\rho_4^*,\rho_5^*) = (\sqrt{Q} - \sqrt{P})^2$. Hence, 
\begin{align*}
    \mathbb E\Big[(X_1 - \mathbb E\big[X_1\big|W_1,W_2,Y_1\big]\big)^2\Big] = \frac{N\cdot(\sqrt{Q} - \sqrt{P})^2}{N+(\sqrt{Q} - \sqrt{P})^2}.
\end{align*}

Obviously, the mimimal estimation cost value $N$ from case 1 is always larger than the cost derived in case 2. Therefore, summarizing our above analysis, the optimal Gaussian cost $S_{\mathsf {G}}(P)$ is given by \eqref{opt gaussian cost}.
\end{proof}

\begin{proof}[Proof of Lemma \ref{lemma: markov-covariance}]
From the Markov chain, we have that
\begin{align*}
    0&=I(X;Z|Y)\\
    &= H(X, Y) + H(Y,Z) - H(Y) - H(X,Y,Z)\\
    &= \frac{1}{2}\log\brackets{\frac{\det(\Sigma_{X,Y})\cdot\det(\Sigma_{Y,Z})}{\sigma_Y^2\cdot\det(\Sigma_{X,Y,Z})}},
\end{align*}
where all the information of the last step can be obtained from the covariance matrix \eqref{eq: cov xyz}. Therefore,
\begin{align*}
&0=\det(\Sigma_{X,Y})\cdot\det(\Sigma_{Y,Z}) - \sigma_Y^2\cdot\det(\Sigma_{X,Y,Z})\\
&= PQ^2V(\rho_1\rho_3 - \rho_2)^2,
\end{align*}
which implies the result \eqref{eq: rho_2=rho_1rho_3}.
\end{proof}

\begin{proof}[Proof of Corollary \ref{cor: causal cost with feedback}]
According to \cite{zhao2024causal}, the information constraint for Witsenhausen counterexample with causal encoder and non-causal decoder with channel feedback $Y_1$ is
\begin{equation}
    I(W_1; Y_1) - I(U_2; X_0|W_1, Y_1)\geq 0. \label{eq: info result c-n w-f}
\end{equation}
Moreover, in the Gaussian settings, the MMSE estimator of $X_1$ can be represented as $U_2 = \mathbb E\big[X_1\big|W_1,W_2,Y_1\big]=a\cdot W_1 + b\cdot W_2+c\cdot Y_1$, with some constants $a,b,c\in\mathbb R$. Hence, the information constraint \eqref{eq: info result c-n w-f} can be rewritten as
    \begin{align*}
        &I(W_1; Y_1) - I(U_2; X_0|W_1, Y_1)\\
        &=I(W_1; Y_1) - I(a\cdot W_1 + b\cdot W_2+c\cdot Y_1; X_0|W_1, Y_1)\\
        &=I(W_1; Y_1) - I(W_2; X_0|W_1, Y_1),
    \end{align*}
which recovers the information constraint for the framework without channel feedback \eqref{eq: info constraint, w/o feedback}. Thus, the optimization domain for this optimization problem is exactly the same as the one without channel feedback, i.e., \eqref{eq: opt domain w/o feedback}.
\end{proof}

\bibliographystyle{ieeetr}
\bibliography{bibliography/IEEEabrv,main}

\end{document}

%% file: fig/Wits-Figure.tex


\begin{tikzpicture}[scale=0.91]
    \draw (2,0) rectangle (3,1);
    \draw (6.8,0) rectangle (7.8,1);

    \draw (4.2,0.5) circle (0.2) node {$+$};
    \draw (5.6,0.5) circle (0.2) node {$+$};

    \filldraw (1,1.5) circle (2pt) node[above] {$X_{0,t}\sim \mathcal{N}(0,Q)$};
    \filldraw (5.6,1.5) circle (2pt) node[above] {$Z_{1,t}\sim \mathcal{N}(0,N)$};

    \draw[->] (1,1.5) -- (1,0.5) -- (2,0.5);
    \draw[->] (1,1.5) -- (4.2,1.5) -- (4.2,0.7);
    \draw[->] (3,0.5) -- (4,0.5);
    \draw[->] (4.4,0.5) -- (5.4,0.5);
    \draw[->] (4.9,0.5) -- (4.9,-0.5) -- (8.8,-0.5);
    \draw[->] (5.6,1.5) -- (5.6,0.7);
    \draw[->] (5.8,0.5) -- (6.8,0.5);
    \draw[->] (7.8,0.5) -- (8.8,0.5);

    \node at (1.5,0.8) {$X_0^t$};
    \node at (3.5,0.8) {$U_{1,t}$};
    \node at (4.9,0.8) {$X_{1,t}$};
    \node at (6.3,0.8) {$Y_{1,t}$};
    \node at (8.3,0.8) {$U_{2}^n$};
    \node at (8.3,-0.2) {$X_{1}^n$};
    \node at (2.5,0.5) {$C_1$};
    \node at (7.3,0.5) {$C_2$};
\end{tikzpicture}
\label{fig:model}


%% file: fig/MMSE_curves.tex
\definecolor{antiquefuchsia}{rgb}{0.57,0.36,0.51}
\definecolor{antiquebrass}{rgb}{0.7,0.48,0.36}
\definecolor{alizarin}{rgb}{0.82,0.1,0.26}
\definecolor{airforceblue}{rgb}{0.36,0.64,0.50}
\definecolor{amethyst}{rgb}{0.6,0.4,0.8}

\begin{tikzpicture}[scale = 0.90]
\begin{axis}[
    xlabel={$P$},
    ylabel={$\mathsf {MMSE}$},
    legend pos=north east,
    axis lines=middle, 
    xlabel style={at={(ticklabel* cs:1)}, anchor=north west},
    ylabel style={at={(ticklabel* cs:1)}, anchor=south east},
    xmin=0, xmax=0.8, 
    ymin=0, ymax=0.09, 
    xtick={0.0172,0.5828}, 
    ytick=\empty, 
    xticklabels={$P_1$, $P_2$}, 
]
\addplot[airforceblue, very thick] table [col sep=comma, x index=0, y index=1] {Tikz/SlP_data.csv};
\addlegendentry{$S_{\ell}(P)$}

\addplot[antiquebrass, very thick] table [col sep=comma, x index=0, y index=1] {Tikz/SgP_data.csv};
\addlegendentry{$S_{\mathsf {G}}(P)$}


\draw [dotted] (axis cs:0,0.0854) -- (axis cs:0.5828, 0.0854);
\draw [dotted] (axis cs:0,0.0146) -- (axis cs:0.5828, 0.0146);
\draw [dotted] (axis cs:0.0172, 0) -- (axis cs:0.0172, 0.0854);
\draw [dotted] (axis cs:0.5828, 0) -- (axis cs:0.5828, 0.0854);


\addplot[alizarin, very thick] table [col sep=comma, x index=0, y index=1] {Tikz/S2P_data.csv};
\addlegendentry{$S_2(P)$}

\addplot[amethyst, very thick] table [col sep=comma, x index=0, y index=1] {Tikz/SdpcP_data.csv};
\addlegendentry{$S_{\mathsf {dpc}}(P)$}

\addplot[only marks, mark=*, mark options={scale=0.5}, color=black] coordinates {(0.0172, 0) (0.5828, 0)(0.0172, 0.0854)(0.5828, 0.0146)};
\end{axis}
\end{tikzpicture}